\documentclass[12pt,a4paper]{amsart}
\usepackage{amsaddr}
\usepackage{amsfonts}
\usepackage{amsthm}
\usepackage{amsmath}
\usepackage{xcolor}
\usepackage{amssymb}
\usepackage{mathtools} \usepackage{amscd}
\usepackage{dsfont} \usepackage[latin2]{inputenc}
\usepackage{t1enc}
\usepackage[mathscr]{eucal}
\usepackage{indentfirst}
\usepackage{graphicx}
\usepackage{graphics}
\usepackage{hyperref}
\usepackage{epic}
\numberwithin{equation}{section}
\usepackage[margin=2.9cm, marginparwidth=2.4cm]{geometry}
\usepackage{epstopdf} 
\usepackage{enumerate}
\usepackage[sort,numbers]{natbib}

\theoremstyle{plain}
\newtheorem{Th}{Theorem}[section]
\newtheorem{Lemma}[Th]{Lemma}
\newtheorem{Cor}[Th]{Corollary}
\newtheorem{Prop}[Th]{Proposition}

 \theoremstyle{definition}

\newtheorem*{Rem*}{Remark}
\newtheorem*{Def*}{Definition}
\newtheorem{?}[Th]{Problem}

\newcommand{\E}[1]{\mathbb{E} #1 }

\newcommand{\Ind}[1]{\mathds{1}_{\{ #1\}}}

\newcommand{\R}{\mathbb{R}}
\newcommand{\N}{\mathbb{N}}
\newcommand{\KLD}{\mathrm{D}_{\mathrm{KL}}}
\newcommand{\econst}{\mathrm{e}}
\newcommand{\dd}{\,\mathrm{d}}
\newcommand{\cd}{\operatorname{CD}}

\usepackage{ulem}
\usepackage{cancel}

\newcommand{\half}{\frac{1}{2}}
\newcommand{\thalf}{\tfrac{1}{2}}

\newcommand{\fg}{\mathfrak{g}}
\newcommand{\cN}{\mathcal{N}}

\begin{document}

\title{Asymptotics of Smoothed Wasserstein Distances}

\author{Hong-Bin Chen \and Jonathan Niles-Weed}
\thanks{JNW gratefully acknowledges the support of the Institute for Advanced Study, where a portion of this research was conducted.}
\address{Courant Institute of Mathematical Sciences, New York University}
\email{hbchen@cims.nyu.edu}
\email{jnw@cims.nyu.edu}

\date{\today}
\begin{abstract}
We investigate contraction of the Wasserstein distances on $\R^d$ under Gaussian smoothing.
It is well known that the heat semigroup is exponentially contractive with respect to the Wasserstein distances on manifolds of positive curvature; however, on flat Euclidean space---where the heat semigroup corresponds to smoothing the measures by Gaussian convolution---the situation is more subtle.
We prove precise asymptotics for the $2$-Wasserstein distance under the action of the Euclidean heat semigroup, and show that, in contrast to the positively curved case, the contraction rate is always polynomial, with exponent depending on the moment sequences of the measures.
We establish similar results for the $p$-Wasserstein distances for $p \neq 2$ as well as the $\chi^2$ divergence, relative entropy, and total variation distance.
Together, these results establish the central role of moment matching arguments in the analysis of measures smoothed by Gaussian convolution.
\end{abstract}

\maketitle

\section{Introduction}
Given two probability distributions $\mu$ and $\nu$ on a Riemannian manifold $M$, what can be said about the Wasserstein distance $W_2^2(\mu P_t, \nu P_t)$, where $P_t$ is the heat semigroup?
The seminal works of \citet{OttVil00} and~\citet{Sturm_Renesse} show that this question is intimately related to the geometry of $M$---in particular, its curvature.
Specifically, \citet{Sturm_Renesse} show that the Ricci curvature of $M$ is bounded below by $K$ if and only if
\begin{equation}\label{curvature_contraction}
W_2(\mu P_t, \nu P_t) \leq \econst^{-Kt} W_2(\mu, \nu)
\end{equation}
for all probability measures $\mu$ and $\nu$ on $M$ and $t \geq 0$. 
In particular, when $M$ is positively curved, the convergence of $W_2(\mu P_t, \nu P_t)$ to zero is exponentially fast.

If we specialize to the flat space $M = \R^d$, then the application of the heat semigroup is nothing more than convolution by the Gaussian measure $\rho_t$ with density
\begin{equation}\label{eqn:rho_def}
\rho_t(x) = \frac{1}{(2\pi t)^\frac{d}{2}}e^{-\frac{1}{2t}|x|^2}, \quad x\in \R^d.
\end{equation}
It is immediate that $W_2(\mu * \rho_t, \nu * \rho_t) \leq W_2(\mu, \nu)$, but since $\R^d$ has zero curvature, \eqref{curvature_contraction} does not imply any strict contraction as $t \to \infty$.
And, indeed, there may be none: if $\mu = \delta_x$ and $\nu = \delta_y$ for $x, y \in \R^d$, then
\begin{equation*}
W_2(\mu * \rho_t, \nu * \rho_t) = W_2(\mu, \nu) = |x - y| \quad \forall t \geq 0\,.
\end{equation*}
If $x \neq y$, then we do not even have $W_2(\mu * \rho_t, \nu * \rho_t) \to 0$.
More generally, it is straightforward to see that if $\mu$ and $\nu$ have different means, then $W_2(\mu * \rho_t, \nu * \rho_t)$ is bounded away from $0$ as $t \to \infty$.

The fact that~\eqref{curvature_contraction} is uninformative on $\R^d$ is well known and has spurred an interest in refinements for finite-dimensional flat spaces.
\Citet{Bolley2014} performed a careful analysis of the heat semigroup on $\R^d$ and established an elegant improvement of~\eqref{curvature_contraction}:
\begin{equation}\label{bolley_bound}
W_2^2(\mu * \rho_t, \nu * \rho_t) \leq W_2^2(\mu, \nu) - \frac{2}{d} \int_0^t (h(\mu * \rho_s) - h(\nu * \rho_s))^2 \dd s\,,
\end{equation}
where $h$ is the differential entropy (i.e., the relative entropy with respect to the Lebesgue measure).
Unlike~\eqref{curvature_contraction}, this result can yield strict contraction even in the absence of curvature.
However, \eqref{bolley_bound} does not make it easy to answer questions of the following type:
\begin{enumerate}
\item Under what conditions on $\mu$ and $\nu$ does $W_2(\mu * \rho_t, \nu * \rho_t) \to 0$?
\item If $W_2(\mu * \rho_t, \nu * \rho_t) \to 0$, at what rate does this contraction occur?
\end{enumerate}

In this work, we give sharp answers to both questions.
A consequence of our main theorem is that, under suitable tail bounds, the quantity $W_2(\mu * \rho_t, \nu * \rho_t)$ always approaches zero as $t \to \infty$ if $\mu$ and $\nu$ have the same mean, but that this convergence always happens at a polynomial rather than exponential rate.
Indeed, if the first $n$ moments of $\mu$ and $\nu$ match but their $(n+1)$th moments do not, then $W_2^2(\mu * \rho_t, \nu * \rho_t) = \Theta(t^{-n})$.
Moreover, we show that the rescaled quantity $t^n W_2^2(\mu * \rho_t, \nu * \rho_t)$ has a positive limit as $t \to \infty$:
\begin{equation*}
\lim_{t \to \infty} t^n W_2^2(\mu * \rho_t, \nu * \rho_t) = c_{\mu, \nu} > 0\,,
\end{equation*}
where $c_{\mu, \nu}$ is an explicit positive constant depending on the $(n+1)$st moments of $\mu$ and $\nu$.
We complement these results by showing that, up to a trivial rescaling, $c_{\mu, \nu}$ is also the limiting value of the relative entropy and $\chi^2$ divergence between $\mu * \rho_t$ and $\nu * \rho_t$.
We establish similar results for the total variation distance, which decays at the same rate but possesses a different limiting value.
Together, these results imply that a variety of measures of discrepancy between probability distributions on $\R^d$ agree in the limit under application of Gaussian smoothing.

Our results also extend to the $p$-Wasserstein distances for $p \neq 2$.
For example, if the first $n$ moments of $\mu$ and $\nu$ match but the $(n+1)$th moments do not and the distributions satisfy sufficiently strong tail conditions, we obtain bounds of the form
\begin{equation*}
0 <  \liminf_{t \to \infty} t^{n/2} W_p(\mu * \rho_t, \nu * \rho_t)\leq \limsup_{t \to \infty} t^{n/2} W_p(\mu * \rho_t, \nu * \rho_t)  < \infty\,.
\end{equation*}

Together with the sharp asymptotics we obtain for $p = 2$, these results show that all the Wasserstein distances $W_p$ for $p \in [1, \infty)$ decay at the same, polynomial rate, under a simple moment matching condition.

\subsection{Related work}
The Wasserstein contraction of diffusion semigroups is deeply connected to several modern areas of geometry and probability theory.
\Citet{OttVil00} first established a link between contraction of the heat semigroup, Talagrand's inequality, and log-Sobolev inequalities.
The connection between these ideas is the formal understanding due to \citet{Otto2001} of the heat semigroup as a gradient flow associated with the entropy functional on the space of probability measures equipped with the Wasserstein metric~\citep{Ambrosio2005}.
This powerful analogy reveals the central role of the convexity of the entropy functional along geodesics in this space and forms the basis for the synthetic notions of Ricci curvature~\citep{Lott2009,Sturm2006,Sturm2006a,Villani2016}.
This perspective also sheds new light on the concentration of measure phenomenon~\citep{Ledoux2001}, via the transportation-entropy inequalities developed by~\citet{Mar96,Mar96a} and~\citet{Tal96}.
More generally, these ideas parallel the development of a general set of techniques for studying Riemannian manifolds via diffusion processes~\citep{Wang2014}.

The condition in~\eqref{curvature_contraction} that the Ricci curvature be bounded below by $K$ is known as the $\cd(K,\infty)$ condition.
The general $\cd(K, N)$ (``curvature-dimension'') condition expresses in a certain sense that the Ricci curvature is bounded below by $K$ and the dimension is at most $N$~\citep{Bakry2014}.
The result of \citet{Bolley2014} given in~\eqref{bolley_bound} is the correct analogue of~\eqref{curvature_contraction} for spaces satisfying the $\cd(0, d)$ condition.
\Citet{Bolley2014} develop similar contractive results involving a different measure of distance for spaces satisfying $\cd(K, N)$ for general $K$ and $N$.
Establishing the correct $\cd(K, N)$ analogues for statements first formulated under a $\cd(K, \infty)$ condition is an area of active research~\citep[see][and references therein]{Bolley2018}.
This line of work is closely related to the problem of proving similar contractive estimates for general diffusion processes \citep{Zhang2018,Luo2014,Wang2016,Eberle2016,Eberle2011}.

To obtain asymptotics for the $2$-Wasserstein distance, we employ a technique similar to one recently used to establish sharp limiting constants for the $2$-dimensional matching problem~\citep{AmbStrTre19}.
We solve a linearized form of the Monge-Amp\`ere equation to obtain a candidate feasible transport solution in the form of a coupling investigated by~\citet{Mos65}.
Evaluating the cost of this \textit{ansatz}---and establishing a matching lower bound via a strategy developed by \citet{Pey18}---shows that this coupling is asymptotically optimal.

The behavior of smoothed version of the Wasserstein distance is also of statistical interest.
Several recent works in statistics and information theory examine the behavior of $W_2^2(\mu * \rho_t, \nu * \rho_t)$ when $\nu = \mu_n$ is an empirical measure comprising $n$ i.i.d.~samples from $\mu$.
\Citet{Weed2018} noticed that when $\mu$ and $\nu$ are compactly supported, the $1$-Wasserstein distance satisfies $\E W_1(\mu * \rho_t, \mu_n * \rho_t) \ll \E W_1(\mu, \mu_n)$ when $t$ is sufficiently large.
This observation implies that certain statistical tasks involving the Wasserstein distance become easier if samples are first smoothed by Gaussian noise.
\Citet{Goldfeld2020} extended this analysis to the total variation distance, relative entropy, and $2$-Wasserstein distance, as well as to distributions with unbounded support.
Motivated by these findings, \citet{Goldfeld2020b} propose to study this smoothed Wasserstein distance as a statistically attractive variant of the standard Wasserstein distances.

Our asymptotic results on the behavior of the $\chi^2$-divergence and relative entropy agree with several nonasymptotic bounds in the statistics literature for Gaussian mixtures~\citep{Bandeira2020,Wu2018}.
To our knowledge, the asymptotic connection with smoothed Wasserstein distances is new.

\subsection{Notation}
We denote by $\N$ the set of nonnegative integers. Sometimes, we shall write $\N\cup\{0\}$ to emphasize that $0$ is admissible.
The symbol $\fg$ will denote the standard Gaussian measure on $\R^d$, namely, $\fg(dx) = \rho_1(x)dx$, where $\rho_1$ is defined as in~\eqref{eqn:rho_def}. For any random variable $Z=(Z_i)_{1\leq i\leq m}$ in $\R^m$ for some $m\in\N\setminus\{0\}$, we denote by $\E Z = (\E Z_i)_{1\leq i\leq m}$ the vector-valued expectation of $Z$.

We recall the following multi-index notation.
For $y \in \R^d$ and $\alpha\in \N^d$, we write~$y^\alpha = y^{\alpha_1}_1y^{\alpha_2}_2\dots y^{\alpha_d}_d$ and $\alpha! = \prod_{i=1}^d (\alpha_i!)$. For $j\in\N$, let  $[j]=\{\alpha \in \N^d: \sum_{i=1}^d\alpha_i = j\}$.

The symbol $c > 0$ denotes a universal constant whose value may change from line to line.
We use subscripts to indicate when such a constant depends on other parameters of the problem.
For real numbers $a$ and $b$, we write $a \vee b$ for $\max\{a, b\}$.

\section{Setting and main results}
Consider two probability measures $\mu$ and $\nu$ on $\R^d$. Let $X$ and $Y$ be random variables with laws $\mu$ and $\nu$, respectively. 
Throughout, we consider measures with sufficiently light tails, which we quantify via the following condition:
\begin{itemize}
    \item Condition $\mathbf{E}(\beta)$, for $\beta>0$, is said to hold if $\E e^{\beta|X-\E X|^2},\ \E e^{\beta|Y-\E Y|^2}<\infty$.
\end{itemize}

The following moment-matching condition plays a central role in our results.
\begin{itemize}
    \item Condition $\mathbf{M}(n)$, for $n\in\N\cup\{0\}$, is said to hold if $n$ is the largest nonnegative integer such that
    \begin{align*}
    \E X^\alpha =\E Y^\alpha,\quad \text{for all } \alpha \in [k] \text{ and all }k\leq n.
\end{align*}
\end{itemize}
In other words, $\mathbf{M}(n)$ holds if  moment tensors of $X$ up to order $n$ match those of $Y$, but those of order $n+1$ do not.

\subsection{Exact asymptotics for the \texorpdfstring{$2$}{2}-Wasserstein distance}\label{sec:exact}
Our first main result gives exact asymptotics for the $2$-Wasserstein distance under the tail condition $\mathbf{E}(\beta)$ and the moment matching condition $\mathbf{M}(n)$.

\begin{Th}\label{Theorem:exact_p=2}
Suppose $\mathbf{E}(\beta)$ holds for some $\beta>0$. If $\mathbf{M}(n)$ holds for some $n\in \mathbb{N}\cup\{0\}$, then
\begin{align*}
    \lim_{t\to\infty} t^{n} W^2_2(\mu*\rho_t, \nu*\rho_t) = \frac{1}{n+1}\sum_{\alpha \in [n+1]}\frac{1}{\alpha!} \Big|\E X^\alpha - \E Y^\alpha \Big|^2.
\end{align*}
\end{Th}

If $n = 0$, then Theorem~\ref{Theorem:exact_p=2} reads
\begin{equation*}
    \lim_{t\to\infty} W^2_2(\mu*\rho_t, \nu*\rho_t) = |\E X - \E Y|^2\,,
\end{equation*}
which recovers exactly the situation identified in the introduction.
Moreover, as Corollary~\ref{corollary:1} below makes clear, this zeroth-order behavior is common to all Wasserstein distances.

We obtain the upper bound in Theorem~\ref{Theorem:exact_p=2} by constructing a coupling between~$\mu*\rho_t$ and $\nu*\rho_t$ via the solution to a PDE obtained by linearizing the Monge-Amp\`ere equation.
The proof of this upper bound appears in Section~\ref{moser}.
To show the lower bound, we employ the concept of displacement interpolation, due to McCann~\cite{McC97}, and control the solution of the PDE considered in Section~\ref{moser} along a geodesic in Wasserstein space.
The proof appears in Section~\ref{section:exact_asym_for_p=2}.

Theorem~\ref{Theorem:exact_p=2} implies several useful estimates, including the following corollary, showing that a good approximation of~$W_2(\mu*\rho_t, \nu*\rho_t)$ can be obtained by replacing $\mu$ and $\nu$ with appropriate Gaussian measures.
\begin{Cor}\label{corollary:2} 
Suppose $\mathbf{E}(\beta)$ holds from some $\beta>0$. Let $\mathcal{N}_\mu$ (respectively, $\mathcal{N}_\nu$) be Gaussian with the same mean and covariance as $\mu$ (respectively, $\nu$). Then we have
\begin{align*}
    \Big|W_2(\mu*\rho_t, \nu*\rho_t)-W_2(\mathcal{N}_\mu*\rho_t,\mathcal{N}_\nu*\rho_t)\Big|=\mathcal{O}(t^{-1}).
\end{align*}
\end{Cor}
Note that the quantity $W_2(\mathcal{N}_\mu*\rho_t,\mathcal{N}_\nu*\rho_t)$ has an explicit expression (see, e.g., \cite[Proposition 7]{givens_Shortt_Gaussian_W_2}).
Since the first two moments of $\mu$ and $\cN_\mu$ (respectively, $\nu$ and $\cN_\nu$) match, Corollary~\ref{corollary:2} follows immediately from Theorem~\ref{Theorem:exact_p=2} after applying the triangle inequality:
\begin{equation*}
\Big|W_2(\mu*\rho_t, \nu*\rho_t)-W_2(\mathcal{N}_\mu*\rho_t,\mathcal{N}_\nu*\rho_t)\Big| \leq W_2(\mu*\rho_t, \mathcal{N}_\mu*\rho_t) + W_2(\nu * \rho_t, \mathcal{N}_\nu*\rho_t) = \mathcal{O}(t^{-1})\,.
\end{equation*}

\subsection{Generalization to \texorpdfstring{$W_p$}{Wp} for \texorpdfstring{$p \neq 2$}{p neq 2}}
Though we do not develop exact asymptotics when $p \neq 2$, our main result of this section shows that the other Wasserstein distances exhibit the same qualitative behavior as $W_2$, namely, that under $\mathbf{M}(n)$ the quantity $W_p(\mu*\rho_t, \nu*\rho_t)$ decays at the rate $t^{-n/2}$.

\begin{Th}\label{Theorem:upper_bound}
Let $n\in\N\setminus\{0\}$ and $p \geq 1$.
If $\mathbf{E}(\beta)$ holds for some $\beta>0$, and $\mathbf{M}(n)$ holds, then there are positive constants $c_{\mu, \nu}$ and $c_{d,n,p,\beta}$ and functions $\underline{h}(t)$ and $\overline{h}(t)$ satisfying $\lim_{t \to \infty} \underline{h}(t) = \lim_{t \to \infty} \overline{h}(t) = 1$ such that
\begin{equation*}
c_{\mu, \nu} \underline{h}(t) \leq t^{\frac n 2} W_p(\mu*\rho_t, \nu*\rho_t) \leq c_{d,n,p,\beta} \overline{h}(t), \qquad \forall t > \frac {p-1} \beta\,.
\end{equation*}
\end{Th}

We show the upper bound in Theorem~\ref{Theorem:upper_bound} in Section~\ref{moser}, where it follows from the same construction used to obtain sharp bounds in the $W_2$ case.
The lower bound follows from simpler ideas and appears in Section~\ref{section:lower_bound}.

Theorem~\ref{Theorem:upper_bound} implies the following corollary, which gives exact zeroth-order asymptotics for $W_p$.
\begin{Cor}\label{corollary:1}
Let $p\geq 1$. Assume $\mathbf{E}(\beta)$ holds for some $\beta>0$. Then we have
\begin{align*}
    \lim_{t\to\infty}W_p(\mu*\rho_t, \nu*\rho_t) = \big|\E X-\E Y\big|.
\end{align*}
\end{Cor}
To obtain Corollary~\ref{corollary:1}, we let $\tilde \mu$ be the law of $\tilde X = X - \E X + \E Y$.
Clearly $W_p(\mu*\rho_t,\tilde \mu*\rho_t)=|\E X -\E Y|$ for all $p\geq 1$ and $t\geq 0$.
The triangle inequality implies
\begin{align*}
    \Big|W_p(\mu*\rho_t,\nu*\rho_t)-|\E X -\E Y|\Big|\leq W_p(\tilde \mu*\rho_t,\nu*\rho_t).
\end{align*}
Applying Theorem \ref{Theorem:upper_bound} to the measures $\tilde \mu$ and $\nu$ yields the claim.

\subsection{Asymptotics for \texorpdfstring{$f$}{f}-divergences}
Theorem~\ref{Theorem:exact_p=2} implies that though the 2-Wasserstein distance is highly nonlinear, its asymptotic behavior under Gaussian smoothing is entirely determined by linear functionals of the measures (i.e., their moments).
In fact, under the same conditions, we show that similar limiting behavior holds for the $\chi^2$ divergence and Kullback--Leibler divergence (relative entropy) between $\mu * \rho_t$ and $\nu * \rho_t$ as well.
Given two probability measures $\mu$ and $\nu$ with Lebesgue densities, recall
\begin{align*}
\chi^2(\mu,\nu) & = \int \bigg(\frac{\mu}{\nu}\bigg)^2\nu dx -1 = \int \frac{(\mu - \nu)^2}{\nu}dx. \\
\KLD(\mu\| \nu) & = \int \log \bigg(\frac{\mu}{\nu}\bigg) \mu dx\,.
\end{align*}
The $\chi^2$ and Kullback-Leibler divergence, as well as the total variation distance defined below, are examples of $f$-divergences, which are common measures of dissimilarity in information theory and statistics~\cite{Csi63,LieVaj06}
The following theorem shows that these divergences have the same asymptotic form as the squared $2$-Wasserstein distance, but decay at the rate $t^{-(n+1)}$ rather than $t^{-n}$.

\begin{Th}\label{Prop:chi^2_relative_entropy}
If $\mathbf{E}(\beta)$  and $\mathbf{M}(n)$ hold for some $\beta>0$ and some $n \in \N\cup \{0\}$, then
\begin{align}
    \lim_{t\to\infty }t^{n+1}\chi^2(\mu*\rho_t,\nu*\rho_t)& =\sum_{\alpha\in[n+1]}\frac{1}{\alpha!}\Big|\E X^\alpha - \E Y^\alpha\Big|^2; \label{eq:exact_asymptotics_chi^2}\\
    \lim_{t\to\infty }t^{n+1} \KLD(\mu*\rho_t\| \nu*\rho_t) & = \frac{1}{2}\sum_{\alpha\in[n+1]}\frac{1}{\alpha!}\Big|\E X^\alpha - \E Y^\alpha\Big|^2 \label{eq:exact_asymptotics_KL-div}.
\end{align}
\end{Th}

We note that~\eqref{eq:exact_asymptotics_KL-div} combined with Theorem~\ref{Theorem:exact_p=2} implies that
\begin{equation}\label{eq:asymptotic_talagrand}
W_2^2(\mu*\rho_t,\nu*\rho_t) \sim \frac{2 t}{n+1} \KLD(\mu*\rho_t\| \nu*\rho_t) \quad \text{as $t \to \infty$.}
\end{equation}
This can be compared with Talagrand's inequality~\cite{Tal96}, which states that the measure $\rho_t$ satisfies
\begin{equation*}
W_2^2(\mu, \rho_t) \leq 2 t \KLD(\mu \| \rho_t) ,\quad \forall \mu\,.
\end{equation*}
Equation \eqref{eq:asymptotic_talagrand} says that $\mu*\rho_t$ and $\nu*\rho_t$ asymptotically enjoy a similar bound.

Finally, we prove exact asymptotics for the total variation distance, defined by
\begin{equation*}
d_{\mathrm{TV}}(\mu, \nu) = \int |\mu - \nu| dx\,.
\end{equation*}
In contrast to the asymptotics for $W_2$, $\chi^2$, and $\KLD$, which have an $L^2$ flavor, the asymptotic behavior of $d_{\mathrm{TV}}$ is governed by $L^1$.

\begin{Th}\label{Theorem:asymptotic_expansion_TV}
Suppose that $\mu$ and $\nu$ have finite $(n+2)$th moments and $\mathbf{M}(n)$ holds for some $n\in\N\cup\{0\}$. Then 
\begin{align*}
    \lim_{t\to\infty} t^\frac{n+1}{2}d_{\mathrm{TV}}(\mu*\rho_t,\nu*\rho_t) = \frac{1}{2}\int \bigg|\sum_{\alpha \in [n+1]}\frac{1}{\alpha!}\Big(\E X^\alpha - \E Y^\alpha\Big)H_\alpha(x)\bigg|\fg(dx)\,,
\end{align*}
where $H_\alpha$ is the Hermite polynomial defined by
\begin{alignat}{2}
H_\alpha(x) & = \prod_{i=1}^d H_{\alpha_i} (x_i)\,, \quad \quad && x \in \R^d,\ \alpha \in \N^d,\label{eq:def_H_alpha}\\
H_m(x) & = (-1)^m e^{\frac{x^2}{2}} \frac{d^m}{dx^m} e^{- \frac{x^2}{2}}\,, \quad && x \in \R,\ m \in \N\,.\label{eq:def_hermite_poly}
\end{alignat}
\end{Th}
As a consequence of the fact that the Hermite polynomials form an orthogonal basis for $L^2(\R^d, \mathfrak g)$, the condition $\mathbf{M}(n)$ implies that the integral in Theorem~\ref{Theorem:asymptotic_expansion_TV} is strictly positive.
Note that Theorem~\ref{Theorem:asymptotic_expansion_TV} does not require assuming that $\mu$ and $\nu$ satisfy the exponential tail condition $\mathbf{E}(\beta)$ for any nonzero $\beta$.

Proofs of Theorems~\ref{Prop:chi^2_relative_entropy} and~\ref{Theorem:asymptotic_expansion_TV} appear in Section~\ref{section:chi^2,KL-div}.

\section{Upper bounds via the Moser coupling}\label{moser}
The goal of this section is to prove upper bounds on the quantity $W_p(\mu * \rho_t, \nu * \rho_t)$ by exhibiting a particular coupling between $\mu * \rho_t$ and $\nu * \rho_t$.
This coupling is obtained by a method due to~\citet{Mos65}.
We assume throughout this section that $\mathbf{M}(n)$ holds for some $n \in \N \setminus \{0\}$ and will handle the case $\mathbf{M}(0)$ separately.
In particular, we assume in this section without loss of generality that
\begin{align}\label{eq:assumption_mean_0_EX=EY=0}
    \E X= \E Y=0.
\end{align}

We first show how to motivate the Moser coupling on a purely heuristic level, following~\citet{CarLucPar14}.
If we assume the existence of a suitably regular map $T$ pushing $\mu * \rho_t$ to $\nu*\rho_t$, then this map must satisfy the Monge-Amp\`ere equation:
\begin{equation*}
\mu * \rho_t(x) = \nu*\rho_t(T(x)) \, \mathbf J_T(x)\,,
\end{equation*}
where $\mathbf J_T$ is the Jacobian determinant of $T$ at $x$.
Let us linearize this equation by assuming that $\mu*\rho_t$ and $\nu*\rho_t$ are close to $\rho_t$, so that we can write $\mu * \rho_t = (1 + \delta_\mu)\rho_t $ and $\nu * \rho_t = (1 + \delta_\nu)\rho_t$, where $\delta_\mu$ and $\delta_\nu$ are small.
Under the additional assumption that $T(x) = x + \delta_T(x)$ for a small perturbation $\delta_T(x)$, we have the approximation $\mathbf J_T(x) \approx 1 + \nabla \cdot \delta_T(x)$.
Combining these approximations yields the first-order expansion
\begin{equation*}
\nabla \cdot \delta_T(x) + (\nabla \log \rho_t(x)) \cdot \delta_T(x) = \delta_\mu(x) - \delta_\nu(x) = (\mu*\rho_t(x) - \nu*\rho_t(x)) \rho_t^{-1}(x)\,.
\end{equation*}
Brenier's theorem suggests writing $\delta_T = \nabla u$ for some $u: \R^d \to \R$.
Using the definition of $\rho_t$, we obtain
\begin{align}\label{eq:PDE_for_u}
    \Delta u(x) - t^{-1}x \cdot \nabla u(x) = (\mu*\rho_t(x)-\nu*\rho_t(x))\rho_t^{-1}(x)\,.
\end{align}
The Moser coupling is finally defined by using a solution to~\eqref{eq:PDE_for_u} to construct a vector field that evolves $\mu*\rho_t$ into $\nu*\rho_t$.

In the remainder of this section, we first give the rigorous details of this construction.
We then prove upper bounds on the cost of this coupling and, in the special case $ p = 2$, develop exact asymptotics.

\subsection{Construction of the Moser coupling}
For each fixed $t > 0$, we study the key equation~\eqref{eq:PDE_for_u}.
There is a weak solution $u$ to this equation satisfying $u\in C^{1}_{\mathrm{loc}}(\R^d)$ (see Lemma \ref{lemma:w_properties} below).
Hence, $\nabla u$ makes sense pointwise.
We now show how to use such a solution to construct a coupling.

For $s\in[0,1]$, define the linear interpolation
\begin{align}\label{eq:def_m_s}
    m_s = (1-s)(\mu*\rho_t)+s(\nu*\rho_t)
\end{align}
and the vector field
\begin{align*}
    \xi_s(x)=\frac{\rho_t\nabla u(x)}{m_s}.
\end{align*}
Using \eqref{eq:PDE_for_u}, one can check
\begin{align*}
    \partial_s m_s + \nabla \cdot (m_s\xi_s)=0.
\end{align*}
This allows us to apply the Benamou--Brenier formula \cite[Corollary 3.2 and Remark 3.3]{brasco2012survey} to obtain
\begin{align}\label{eq:W_2_est_BB_formula}
\begin{split}
        W^p_p(\mu*\rho_t, \nu*\rho_t)&\leq \int_0^1\int_{\R^d}|\xi_s(x)|^p m_s(x)dx ds\\ & =\int_{\R^d}|\nabla u(x)|^p\bigg(\int_0^1 \Big(\frac{\rho_t(x)}{m_s(x)}\Big)^{p-1}ds\bigg) \rho_t(x) dx.
\end{split}
\end{align}

To estimate the term inside parentheses in the above display, we first establish a lower bound for $\mu*\rho_t$ (and similarly for $\nu*\rho_t$):
\begin{align}\label{eq:mu*rho_t_lower_bound}
\begin{split}
    \mu*\rho_t(x)&= (2\pi t)^{-\frac{d}{2}}\E e^{-\frac{1}{2t}|x-X|^2}\geq (2\pi t)^{-\frac{d}{2}} e^{-\frac{1}{2t}\E|x-X|^2}\\
    &=\rho_t(x) e^{\frac{1}{t}\langle x, \E X\rangle-\frac{1}{2t}\E |X|^2}=\rho_t(x) e^{-\frac{1}{2t}\E|X|^2}.
\end{split}
\end{align}
Here, we used Jensen's inequality and the assumption \eqref{eq:assumption_mean_0_EX=EY=0}. 
By \eqref{eq:def_m_s},  \eqref{eq:mu*rho_t_lower_bound} and a similar lower bound for $\nu*\rho_t$, we obtain
\begin{align*}
    m_s(x)\rho_t^{-1}(x)\geq e^{-\frac{1}{2t}(\E|X|^2\vee \E|Y|^2)}, \quad x\in\R^d,
\end{align*}
which then gives
\begin{align*}
    \int_0^1 \Big(\frac{\rho_t(x)}{m_s(x)}\Big)^{p-1}ds\leq  e^{\frac{p-1}{2t}(\E|X|^2\vee \E|Y|^2)}, \quad x\in\R^d.
\end{align*}
Plug this into \eqref{eq:W_2_est_BB_formula} and apply a change of variables to see
\begin{align}\label{eq:upper_bound_p>1_step_2}
    W^p_p(\mu*\rho_t, \nu*\rho_t) & \leq e^{\frac{p-1}{2t}(\E|X|^2\vee \E|Y|^2)}\int_{\R^d}|\nabla u(x)|^p \rho_t(x) dx \nonumber\\
    & =  e^{\frac{p-1}{2t}(\E|X|^2\vee \E|Y|^2)} \int |\nabla u(t^\half x)|^p \fg(dx).
\end{align}

To evaluate the integral appearing in~\eqref{eq:upper_bound_p>1_step_2}, we first define some notation.
The following objects will appear many times in this paper:
\begin{align}\label{eq:def_eta}
\begin{split}
    \eta(x, y) &=\exp\big(\langle x, y\rangle - \thalf|y|^2\big),\\
    \Theta_t(x)&=t^\half\Big(\E \eta(x, t^{-\frac{1}{2}}X)-\E \eta(x, t^{-\frac{1}{2}}Y)\Big).
\end{split}
\end{align}
By this definition, we immediately have
\begin{align}\label{eq:formula_difference_densities}
    (\mu*\rho_t-\nu*\rho_t)(x) = t^{-\half }\Theta_t(t^{-\half}x)\rho_t(x), \quad x\in\R^d.
\end{align}
Let us introduce
\begin{align}\label{eq:def_w}
    w(x)=t^{-\half} u (t^\half x), \quad x\in\R^d
\end{align}
and the Ornstein--Uhlenbeck operator
\begin{align}\label{eq:def_L_OU_operator}
    L=\Delta - x\cdot \nabla
\end{align}
Hence, due to \eqref{eq:formula_difference_densities} and \eqref{eq:PDE_for_u}, we know $w$  solves
\begin{align}\label{eq:Lw_eq}
    L w =\Theta_t\,.
\end{align}

Adopting the notation from the Malliavin calculus, we write the gradient $\nabla$ as $D$. Under this notation, \eqref{eq:upper_bound_p>1_step_2} becomes
\begin{align}\label{eq:BB_after_change_of_var}
    W^p_p(\mu*\rho_t, \nu*\rho_t) \leq  e^{\frac{p-1}{2t}(\E|X|^2\vee \E|Y|^2)} \int |D w|^pd\fg.
\end{align}
Our upper bounds follow from analysis of~\eqref{eq:BB_after_change_of_var}.
We first show how to obtain an upper bound of the right order when $p \neq 2$ before performing a more careful argument for the $p = 2$ case.

\subsection{A general upper bound}
Our first bound shows that~\eqref{eq:BB_after_change_of_var} is of order $t^{-np/2}$ for any $p \geq 1$. 
This will suffice to prove the upper bound of Theorem~\ref{Theorem:upper_bound}.

We introduce the following notion. For $k\in\N$, let $\mathbb{D}^{k,p}$ be the completion of smooth functions on $\R^d$,  whose derivatives grow at most polynomially, under the norm
\begin{align}\label{eq:sobolev_norm}
    \|\cdot\|_{k,p}=\bigg(\sum_{j=0}^k\int |D_j\cdot|^pd\fg\bigg)^\frac{1}{p},
\end{align}
where $D_j =\partial_j$. The space $\mathbb{D}^{k,p}$ is the Sobolev space with $\fg$ as the underlying measure. We collect useful properties of $w$ in the lemma stated below, the proof of which is given in Section \ref{section:proof_of_lemma_w_properties}.
\begin{Lemma}\label{lemma:w_properties}
Under the assumption \eqref{eq:assumption_mean_0_EX=EY=0}, there is a weak solution $w$ to \eqref{eq:Lw_eq}, which also satisfies 
\begin{enumerate}
    \item \label{item:1_w} $w\in \mathbb{D}^{2,p}$ for all $p\in[1,\infty)$, and $w\in C^1_{\mathrm{loc}}(\R^d)$;
    \item \label{item:2_w} the mean of $Dw$ is zero, namely, $\int Dw d\fg =0$.
    \item \label{item:3_w} there is a constant $c_p>0$ depending only on $p$ such that
    \begin{align*}
        \int |D^2w|^p d\fg \leq c_p\int |Lw|^p d\fg.
    \end{align*}
\end{enumerate}
\end{Lemma}

The first two parts of the lemma allow us to apply the Poincar\'e inequality (Lemma \ref{Lemma:Poincare}) to see
\begin{align*}
    \int |Dw|^p d\fg \leq c_p\int |D^2w|^p d\fg.
\end{align*}
Hence, part \eqref{item:3_w} of Lemma \ref{lemma:w_properties} and \eqref{eq:Lw_eq} imply
\begin{align}\label{eq:Dw_bounded_by_Theta_t}
    \int |D w|^p d\fg \leq c_p \int |L w|^pd\fg = c_p \int |\Theta_t|^pd\fg.
\end{align}
Combining this bound with \eqref{eq:upper_bound_p>1_step_2}, we obtain
\begin{align*}
    W^p_p(\mu*\rho_t, \nu*\rho_t) \leq c_p e^{\frac{p-1}{2t}(\E|X|^2\vee \E|Y|^2)} \int |\Theta_t|^pd\fg.
\end{align*}
Finally, to prove the upper bound of Theorem~\ref{Theorem:upper_bound}, we apply the following lemma, whose proof appears in Section~\ref{section:proof_of_p-norm_Theta}.
\begin{Lemma}\label{Lemma:p-norm_of_Theta_t}
Suppose $\mathbf{E}(\beta)$ holds for some $\beta>0$ and $\mathbf{M}(n)$ holds for some $n \in \N$ (assumption \eqref{eq:assumption_mean_0_EX=EY=0} is not needed). For each $p\geq 1$ and each $\delta>1$, there is $c_{d,n,p,\delta,\beta}$ such that
\begin{align*}
    \bigg(\int |\Theta_t|^pd\fg\bigg)^\frac{1}{p}\leq c_{d,n,p,\delta,\beta}t^{-\frac{n}{2}} \max_{Z\in\{X, Y\}}\E e^{ \frac{\delta(p-1)}{2t}|Z|^2}, \quad t>\tfrac{\delta(p-1)}{2\beta}.
\end{align*}
\end{Lemma}

Choosing $\delta = 2$ and letting $\overline{h}(t) = e^{\frac{p-1}{2t}(\E|X|^2\vee \E|Y|^2)} \cdot \max_{Z\in\{X, Y\}}\E e^{ \frac{\delta(p-1)}{2t}|Z|^2}$, we obtain the desired claim.

\subsection{Upper bound for \texorpdfstring{$p = 2$}{p=2}}
To obtain a sharper estimate when $p = 2$, we apply integration by parts to~\eqref{eq:BB_after_change_of_var} to obtain
\begin{align}\label{eq:upper_bound_p=2}
    W^2_2(\mu*\rho_t, \nu*\rho_t)\leq e^{\frac{1}{2t}(\E|X|^2\vee \E|Y|^2)}\int -wLwd\fg.
\end{align}

In Section \ref{section:proof_of_asymptotics_int_wlw}, we prove the following result.
\begin{Lemma}\label{Lemma:asymptotics_int_wLw}
As $t \to \infty$,
\begin{align}\label{eq:asymptotics_int_wLw}
\int -wLw d\fg = \frac{t^{-n}}{n+1}\sum_{\alpha \in [n+1]}\frac{1}{\alpha!} |\E X^\alpha - \E Y^\alpha |^2+\mathcal{O}(t^{-n-1})\,.
\end{align}
\end{Lemma}

Applying this lemma along with the fact that $e^{\frac{1}{2t}(\E|X|^2\vee \E|Y|^2)}$ approaches $1$ as $t \to \infty$ yields the upper bound of Theorem~\ref{Theorem:exact_p=2}.

Recall that we have assumed \eqref{eq:assumption_mean_0_EX=EY=0}.
If $\E X = \E Y = v \neq 0$, then applying the same proof to centered versions of $\mu$ and $\nu$ yields an upper bound which depends on
\begin{equation*}
\frac{1}{n+1}\sum_{\alpha \in [n+1]}\frac{1}{\alpha!} |\E (X - v)^\alpha - \E (Y - v)^\alpha |^2\,.
\end{equation*}
But under the condition $\mathbf{M}(n)$ for $n \geq 1$, we have
\begin{equation*}
\E (X - v)^\alpha - \E (Y - v)^\alpha = \E X^\alpha - \E Y^\alpha \, \quad \forall \alpha \in [n+1]\,.
\end{equation*}
so we recover precisely the desired bound.

Finally, under $\mathbf{M}(0)$, we use the argument of Corollary~\ref{corollary:1} to reduce to the $n \geq 1$ case.

\section{Estimates for solutions to the Ornstein-Uhlenbeck PDE}
In Section~\ref{moser}, we established that good upper bounds for the Wasserstein distances can be obtained by understanding solutions to~\eqref{eq:Lw_eq}, which reads:
\begin{equation*}
Lw = \Theta_t\,,
\end{equation*}
where $L$ is the Ornstein-Uhlenbeck operator $L = \Delta - x \cdot \nabla$ and $\Theta_t$ is defined in~\eqref{eq:def_eta}.

In this section, we derive the key estimates on solutions to~\eqref{eq:Lw_eq}, via which we obtain the bounds given in Section~\ref{moser}.
As we shall see, these estimates also play a role in obtaining good lower bounds, a question we turn to in Sections~\ref{section:exact_asym_for_p=2} and~\ref{section:lower_bound}.

We first establish several preliminaries involving the Malliavin calculus, before giving the promised proof of Lemma~\ref{lemma:w_properties}.
In the remainder of the section, we derive the necessary estimates on $\Theta_t$.

\subsection{Preliminaries}\label{section:preliminaries}
We begin by reviewing several concepts from analysis on Gaussian spaces.

Consider the stochastic process $W=\{W(h):\R^d\to \R\}_{h\in\R^d}$ given by $W(h)(x)=\langle h, x\rangle$. Under the probability measure $\fg$, one can see that $W$ is an (centered) isonormal Gaussian process with variance $\E_\fg W(h)W(g)=\langle h,g\rangle$, where $\E_\fg$ denotes the expectation with respect to $\fg$.

Recall the Hermite polynomials $H_m$ given in \eqref{eq:def_hermite_poly}. For $m\in \N$, let $\mathcal{H}_m$ be the closed linear subspace of $L^2=L^2(\fg)$ generated by $\{H_m(W(h)):\ h\in \R^d, \ |h|=1\}$. The space $\mathcal{H}_m$ is called the $m$th Wiener chaos, and $\{\mathcal{H}_m\}_{m \geq 0}$ forms an orthogonal decomposition of $L^2$. Let $J_m:L^2\to \mathcal{H}_m$ be the orthogonal projection. In particular, for $\varphi \in L^2$, we have $J_0\varphi= \int \varphi d\fg$.

Recall the Sobolev space $\mathbb{D}^{k,p}$ with norm given in \eqref{eq:sobolev_norm}. Let $\mathcal{P}$ denote the set of polynomials on $\R^d$, which is dense in $\mathbb{D}^{k,p}$ for $p>1$ and $k\geq 0$ (see \cite[Corollary 1.5.1 and Excercise 1.1.7]{nualart}). On $\mathcal{P}$, the operator $L$ in \eqref{eq:def_L_OU_operator} can be expressed as $L=\sum_{m=0}^\infty -mJ_m$ and its pseudo-inverse as $L^{-1}=\sum_{m=1}^\infty -\frac{1}{m}J_m$. Still on $\mathcal{P}$, we can define the negative square root of $-L$ by $C=\sum_{m=0}^\infty -\sqrt{m}J_m$ and its pseudo-inverse $C^{-1}=\sum_{m=1}^\infty -\frac{1}{\sqrt{m}}J_m$. For more details, see \cite{nualart}.

Finally, we require the following $L^p$ version of the Poincar\'e inequality for Gaussian measures~\cite[Corollary 2.4]{pisier1986probabilistic}.
\begin{Lemma}\label{Lemma:Poincare}
Let $p\geq1$. There is $c_p>0$ (depending only on $p$) such that, for all $\varphi\in \mathbb{D}^{1,p}$,
\begin{align*}
   \int|\varphi -\bar\varphi|^pd\fg \leq c_p \int |D\varphi|^pd\fg, 
\end{align*}
where $\bar\varphi = \int\varphi d\fg$.
\end{Lemma}

\subsection{Proof of Lemma \ref{lemma:w_properties}}\label{section:proof_of_lemma_w_properties}
\subsubsection*{Part \eqref{item:1_w}}
We first construct a solution to \eqref{eq:Lw_eq} using $L^{-1}$. 

Lemma \ref{Lemma:p-norm_of_Theta_t} implies $\Theta_t\in L^p$. By density, let $\varphi_k\in\mathcal{P}$ be polynomials such that $\varphi_k \to \Theta_t$ in $L^p$ as $k\to \infty$. The mean of $\Theta_t$ is zero because, due to \eqref{eq:assumption_mean_0_EX=EY=0} and \eqref{eq:def_eta},
\begin{align*}
    \int \Theta_t d\fg = t^\half\E\int (\mu*\rho_t-\nu*\rho_t)dx =0.
\end{align*}
Hence, we may assume $\varphi_k$ all have zero means, namely, $J_0\varphi_k=0$. Applying the multiplier theorem \cite[Theorem 1.4.2]{nualart} to $L^{-1}$, and using \cite[Theorem 1.5.1]{nualart} with the relation $-C^2=L$, one can see that the limit
\begin{align}\label{eq:w_approx}
    w = \lim_{k\to\infty} L^{-1}\varphi_k =L^{-1}\Theta_t \quad \text{in }\mathbb{D}^{2,p}
\end{align}
exists and $L:\mathbb{D}^{2,p} \to L^p$ is continuous. Therefore, we have $Lw = \lim_{k\to\infty}L L^{-1}\varphi_k = \lim_{k\to\infty} \varphi_k - J_0\varphi_k = \Theta_t$ in $L^p$.

It remains to check $w \in C^1_{\mathrm{loc}}(\R^d)$. On each Euclidean ball $B\subset \R^d$, the standard Gaussian measure $\fg$ has a density both bounded above and below. Hence, due to $w\in \mathbb{D}^{2,p}$, we know that $w$ also belongs to the standard Sobolev space $W^{2,p}(B)$ for the Lebesgue measure, for all $p\in [1,\infty)$. The standard Sobolev embedding theorem (see, e.g., \cite[part 3 of Theorem 3.26]{giaquinta2013}) implies  $w\in C^1(B) $, and thus $w\in C^1_{\mathrm{loc}}(\R^d)$.

\subsubsection*{Part \eqref{item:2_w}}
Recall that we write $D_i = \partial_i$. Using the approximation \eqref{eq:w_approx} and performing integration by parts for polynomial integrands, we have
\begin{align*}
    \int D_i wd\fg &= \lim_{k\to\infty}\langle D_i L^{-1}\varphi_k,1\rangle_\fg = -\lim_{k\to \infty }\langle  L^{-1}\varphi_k,x_i\rangle_\fg\\
    &= -\lim_{k\to \infty }\langle  \varphi_k,L^{-1}x_i\rangle_\fg=\lim_{k\to \infty }\langle  \varphi_k,x_i\rangle_\fg= \langle  \Theta_t,x_i\rangle_\fg,
\end{align*}
where $\langle\cdot ,\cdot \rangle_\fg$ is the $L^2(\fg)$ inner product. Here in the third equality, we used the self-adjointness of $L^{-1}$ which is evident from its formula on polynomials. In the penultimate equality, we used the fact that $L^{-1}x_i= -J_1x_i = -x_i$ because $x_i$ belongs to the first order Wiener chaos $\mathcal{H}_1$. 
Hence it is sufficient to check $\int x_i \Theta_t d\fg =0$.
Indeed, due to \eqref{eq:assumption_mean_0_EX=EY=0}, we have
\begin{align*}
    \int x_i \E \exp \Big(\langle x, t^{-\frac{1}{2}}X\rangle - \tfrac{1}{2}|t^{-\frac{1}{2}}X|^2 \Big) \fg(dx)= \E(2\pi)^{-\frac{d}{2}}\int x_i e^{-\frac{1}{2}|x-t^{-\frac{1}{2}}X|^2} dx=\E t^{-\frac{1}{2}}X_i =0,
\end{align*}
and a similar equality with $X$ replaced by $Y$. Finally, by \eqref{eq:def_eta}, we conclude that $\int x_i \Theta_t d\fg =0$.

\subsubsection*{Part \eqref{item:3_w}}
This is an immediate consequence of \cite[Theorem 1.5.1]{nualart}, the density of $\mathcal{P}$ and the fact $C^2=-L$ on $\mathcal{P}$.

\subsection{Proofs of Some Estimates} \label{section:Lemma:Theta_formula_for_matching_moments}
The bounds in Section~\ref{moser} relied on two estimates: Lemma~\ref{Lemma:p-norm_of_Theta_t}, which showed $\|\Theta_t\|_{L^p(\fg)} = \mathcal{O}(t^{-\frac n 2})$, and Lemma~\ref{Lemma:asymptotics_int_wLw}, which gave exact asymptotics for $- \int w L w \fg$.
In this section, we prove both lemmas.

For a multi-index $\alpha\in \N^d$, we write  $\partial^\alpha=\partial^{\alpha_1}_1\partial^{\alpha_2}_2\dots \partial^{\alpha_d}_d$. All derivatives below are with respect to $y$.

Fix $\alpha\in[j]$ for some $j\in\N$. 
In view of \eqref{eq:def_eta}, to study asymptotics of $\Theta_t$, we shall derive the expansion of $\eta(x,y)$ in $y$ for fixed $x$.
In the following, we express $\partial^\alpha\eta(x,y)$ in terms of Hermite polynomials. Recall our notation for Hermite polynomials is given in \eqref{eq:def_H_alpha} and \eqref{eq:def_hermite_poly}.
It can be checked that
\begin{align*}
    \frac{d^m}{dy^m}e^{xy-\frac{y^2}{2}}=H_m(x-y)e^{xy-\frac{y^2}{2}},\quad x,y\in \R.
\end{align*}
Since $\eta(x,y) = \prod_{i=1}^d (e^{x_iy_i-\frac{y^2_i}{2}})$, the above two displays imply that
\begin{align*}
    \partial^\alpha\eta(x,y)=  H_\alpha(x-y) \eta(x,y),\quad x,y\in\R^d.
\end{align*}

Apply the Taylor expansion (with the remainder expressed as an integral) to $\eta(x,y)$ in $y$ around $0$ to see that, for each $n\in \N$,
\begin{align}\label{eq:eta_expansion}
    \eta(x,y) =\sum^n_{j=0} a_j(x,y)  + r_{n+1}(x,y), \quad x,y\in\R^d,
\end{align}
where 
\begin{align}
    a_j(x,y)&= \sum_{\alpha\in [j]}\frac{y^\alpha}{\alpha !}   \partial^\alpha  \eta(x,0)=\sum_{\alpha\in [j]}\frac{y^\alpha}{\alpha !} H_\alpha (x),\nonumber\\ 
    \begin{split}\label{eq:r_n+1_formula}
        r_{n+1}(x,y)& =(n+1)\int_0^1 (1-s)^n\sum_{\alpha\in[n+1]}\frac{y^\alpha}{\alpha !}   \partial^\alpha \eta(x,sy)ds\\&=(n+1)\int_0^1 (1-s)^n\sum_{\alpha\in[n+1]}\frac{y^\alpha}{\alpha !}   H_\alpha (x-sy)\eta(x,sy)ds. 
    \end{split}
\end{align}

We can now prove the required estimates
\subsubsection{Proof of Lemma~\ref{Lemma:p-norm_of_Theta_t}}\label{section:proof_of_p-norm_Theta}
The condition $\mathbf{M}(n)$ implies that 
\begin{align*}
    \E a_j(x,t^{-\half}X)=\E a_j(x,t^{-\half}Y),\quad x\in\R^d,\  j=0,1,2,\dots,n,
\end{align*}
which together with \eqref{eq:def_eta} and \eqref{eq:eta_expansion} yields the following

\begin{Lemma}\label{Lemma:Theta_formula_for_matching_moments}
Suppose $\mu$ and $\nu$ have finite $(n+1)$th moments and $\mathbf{M}(n)$ holds for some $n\in \N\cup\{0\}$. With $r_{n+1}$ given in \eqref{eq:r_n+1_formula}, we have
\begin{align*}
    \Theta_t(x)=t^\half\Big(\E r_{n+1}(x,t^{-\half}X)-\E r_{n+1}(x,t^{-\half}Y)\Big), \quad x\in\R^d,\ t>0.
\end{align*}
\end{Lemma}

Lemma~\ref{Lemma:p-norm_of_Theta_t} then follows from Lemma~\ref{Lemma:Theta_formula_for_matching_moments} and the following result, whose proof appears in Section~\ref{proof_of_integral_est}

\begin{Lemma}\label{lemma:r_n+1_integral_est}
Let $p\geq 1$. Suppose $\mathbf{E}(\beta)$ holds for some $\beta>0$. Then for each $m\in \N\cup\{0\}$ and $\delta>1$, there is $c_{d,m,p,\delta,\beta}>0$ such that, for $Z\in\{X,Y\}$,
\begin{align}\label{eq:r_m+1_pth_moment_est}
    \int |t^\half\E r_{m+1}(x, t^{-\frac{1}{2}}Z)|^p\fg(dx)
    \leq c_{d,m,p,\delta,\beta}t^{-\frac{mp}{2}}\Big( \E e^{\frac{\delta(p-1)}{2t}|Z|^2} \Big)^p, \quad t>\tfrac{\delta(p-1)}{2\beta}.
\end{align}
If $p=1$, \eqref{eq:r_m+1_pth_moment_est} holds for $t>0$ under a weaker assumption $\E|X|^{m+1},\E|Y|^{m+1}<\infty$.
\end{Lemma}

\subsubsection{Proof of Lemma~\ref{Lemma:asymptotics_int_wLw}}\label{section:proof_of_asymptotics_int_wlw}
Using \eqref{eq:def_eta},  \eqref{eq:eta_expansion} and the assumption $\mathbf{M}(n)$, we can write
\begin{align}\label{eq:Theta_t=Q+R}
\Theta_t = Q+R
\end{align}
where
\begin{align}
\begin{split}\label{eq:Q_formula}
Q(x)&= t^\half \Big(\E a_{n+1}(x,t^{-\half}X) -\E a_{n+1}(x, t^{-\half}Y)\Big)\\
&=t^{-\frac{n}{2}}\sum_{\alpha\in [n+1]}\frac{1}{\alpha!}(\E X^\alpha - \E Y^\alpha) H_\alpha(x),
\end{split}\\
R(x)&= t^\half \Big(\E r_{n+2}(x,t^{-\half}X)-\E r_{n+2}(x,t^{-\half}Y)\Big).\label{eq:R_formula}
\end{align}

To carry out our computations, we need the result that $Q$ and $R$ are orthogonal to each other, namely, 
\begin{align}\label{eq:QR_orthogonal}
\int Q R d\fg =0\,.
\end{align}
We prove this fact in Section~\ref{QR_proof}.

Due to \eqref{eq:w_approx}, we have $w = L^{-1}\Theta_t=L^{-1}(Q+R)$. Since $J_{n+1}Q = Q\in \mathcal{P}$, by the definition of $L$, we have $L^{-1}Q = -\frac{1}{n+1}Q$. Using this, \eqref{eq:QR_orthogonal} and the self-adjointness of $L^{-1}$, we can compute
\begin{align}\label{eq:int_wLw_in_terms_of_Q_R}
\begin{split}
    \int -wLw d\fg & = \int -(L^{-1}Q+L^{-1}R)(Q+R)d\fg\\
    & = \frac{1}{n+1}\int |Q|^2 d\fg -\int R L^{-1}R d\fg.
\end{split}
\end{align}

To determine the first term on the right hand side, we need the following standard fact.

\begin{Lemma}\label{lemma:hermite_poly_orthognoality}
Let $\alpha$ and $\beta$ be two multi-indices. Then
\begin{align*}
    \int H_\alpha H_\beta d\fg =
    \begin{cases}
    \alpha! \quad &\text{if }\alpha =\beta\\
    0 \quad &\text{if }\alpha \neq\beta
    \end{cases}.
\end{align*}
\end{Lemma}

This lemma together with \eqref{eq:Q_formula} gives
\begin{align}\label{eq:int_Q^2}
    \int |Q|^2d\fg = t^{-n}\sum_{\alpha \in [n+1]}\frac{1}{\alpha!}\big|\E X^\alpha - \E Y^\alpha\big|^2.
\end{align}

To estimate the second term, recall the definition of the operator $C^{-1}$ in Section \ref{section:preliminaries} and see that
\begin{align*}
    \int R L^{-1}Rd\fg = \int |C^{-1}R|^2d\fg \leq c\int |R|^2 d\fg,
\end{align*}
where we used \cite[Theorem 1.4.2]{nualart} in the last inequality. Recall the formula \eqref{eq:R_formula}. Then, Lemma \ref{lemma:r_n+1_integral_est}  with $m=n+1$ implies that for each $\delta>1$, there is $c_{d,n,\delta}>0$ such that 
\begin{align}\label{eq:decay_rate_R}
    \int |R|^2 d\fg \leq c_{d,n,\delta} t^{-n-1}\max_{Z\in\{X,Y\}}\Big(\E e^{\frac{\delta}{2t}|Z|^2}\Big)^2 = \mathcal{O}(t^{-n-1}).
\end{align}
Hence, the above display gives $\int RL^{-1}Rd\fg=\mathcal{O}(t^{-n-1})$. Combining \eqref{eq:int_wLw_in_terms_of_Q_R}, \eqref{eq:int_Q^2}, and~\eqref{eq:decay_rate_R} yields the claim.

\section{Exact asymptotics for \texorpdfstring{$p=2$}{p=2}}\label{section:exact_asym_for_p=2}
In Section~\ref{moser}, we exhibited a valid coupling between $\mu*\rho_t$ and $\nu * \rho_t$ achieving the claimed upper bound of Theorem~\ref{Theorem:exact_p=2}.
A key step in the proof was the derivation of inequality~\eqref{eq:upper_bound_p=2}:
\begin{align*}
    W^2_2(\mu*\rho_t, \nu*\rho_t)\leq e^{\frac{1}{2t}(\E|X|^2\vee \E|Y|^2)}\int -wLwd\fg.
\end{align*}
The goal of this section is to prove the following complementary lower bound.
\begin{Prop}\label{prop:p=2_lower_bound}
Assume $\mathbf{E}(\beta)$ and $\mathbf{M}(n)$ hold for some $\beta > 0$ and $n \in \N\setminus\{0\}$, respectively.
Then as $t \to \infty$,
\begin{equation}
W^2_2(\mu*\rho_t, \nu*\rho_t) \geq (1+o(1))\int -wLwd\fg\,.
\end{equation}
\end{Prop}
Combined with the explicit expansion developed in Lemma~\ref{Lemma:asymptotics_int_wLw}, this proposition proves the desired lower bound of Theorem~\ref{Theorem:exact_p=2} under $\mathbf{M}(n)$ for $n \geq 1$.
As in Section~\ref{moser}, the bound for $\mathbf{M}(0)$ follows immediately from the argument of Corolary~\ref{corollary:1}.

We now proceed with the proof.
We first define the displacement interpolation between two measures~\cite{McC97}.
Denote by $\#$ the push-forward of a measure under a map.
Since $\mu*\rho_t$ and $\nu*\rho_t$ are absolutely continuous with respect to the Lebesgue measure, the optimal coupling between them is given by a convex function $\phi$ on $\R^d$ such that $(\nabla \phi)_\#(\mu*\rho_t) = \nu*\rho_t$.
We then define the displacement interpolation $\mu_s$ by
\begin{align}\label{eq:def_mu_s_displacement_interpolation}
    \mu_s = ((1-s)\mathrm{Id} + s \nabla \phi)_\# (\mu*\rho_t),\quad s\in[0,1].
\end{align}

Recall $u$ in \eqref{eq:PDE_for_u}.
Since $u$ is locally Lipschitz, we have the following inequality (see~[\citealp[Lemma A.1]{Lott2009};~\citealp[Lemma 13]{CheMauRig20}]):
\begin{align}\label{eq:W_2_lower_bound_step_0}
    \int -u (\mu*\rho_t - \nu*\rho_t)dx \leq W_2(\mu*\rho_t,\nu*\rho_t)\int_0^1\Big(\int|\nabla u|^2 \mu_s dx\Big)^\half ds.
\end{align}

Let us set
\begin{align}\label{eq:def_a_a(t)}
    a=a(t)=1+t^{-\half}.
\end{align}
We work with the Gaussian measure $\rho_{at}$ which is slight perturbation of $\rho_t$.

The following bound holds
\begin{align}\label{eq:L_infinity_bound}
    \Big\|\frac{\mu_s}{\rho_{at}}\Big\|_\infty \leq \Big\|\frac{\mu*\rho_t}{\rho_{at}}\Big\|_\infty\vee \Big\|\frac{\nu*\rho_t}{\rho_{at}}\Big\|_\infty.
\end{align}
Indeed, by choosing $\rho_{at}$ to be the reference measure on $\R^d$, the curvature-dimensional criterion $\mathrm{CD}(0,\infty)$ (c.f.\ \cite[Theorem 14.8]{villani2009optimal}) is satisfied. This allows us to apply \cite[Theorem 17.15]{villani2009optimal} to see that, for each $p>1$, the following functional is displacement-convex:
\begin{align*}
    U_p(\rho) =
    \begin{cases}
    \int_{\R^d} \big|\frac{\rho}{\rho_{at}}\big|^p\rho_{at}dx, \quad &\text{if $\rho$ is a probability density,}\\
    \infty,  \quad &\text{otherwise.}
    \end{cases}
\end{align*}
Hence, we have
\begin{align*}
     \Big\|\frac{\mu_s}{\rho_{at}}\Big\|_{L^p(\rho_{at})} \leq \Big\|\frac{\mu*\rho_t}{\rho_{at}}\Big\|_{L^p(\rho_{at})}\vee \Big\|\frac{\nu*\rho_t}{\rho_{at}}\Big\|_{L^p(\rho_{at})}.
\end{align*}
Sending $p\to \infty$, we obtain \eqref{eq:L_infinity_bound}.

Then, we estimate the right hand side of \eqref{eq:L_infinity_bound}:
\begin{align}\label{eq:mu*rho_t_upper_bound}
\begin{split}
        \frac{\mu*\rho_t}{\rho_{at}}(x)& = \Big(\frac{2\pi a t}{2\pi t}\Big)^\frac{d}{2}\E \exp\bigg({-\frac{1}{2t}\Big(|x|^2 -2\langle x, X \rangle +|X|^2 -\frac{1}{a}|x|^2\Big)}\bigg)\\
    &=a^\frac{d}{2}\E e^{-\frac{1}{2t}|\sqrt{\frac{a-1}{a}}x-\sqrt{\frac{a}{a-1}}X|^2}e^{\frac{1}{2t(a-1)}|X|^2}\leq (1+t^{-\half})^\frac{d}{2}\E e^{\frac{1}{2}t^{-\half}|X|^2}.
\end{split}
\end{align}
An analogous bound holds for $\nu*\rho_t/\rho_{at}$. Now let
\begin{align*}
    c_1(t)=(1+t^{-\half})^\frac{d}{2}\Big(\E e^{\frac{1}{2}t^{-\half}|X|^2}\vee \E e^{\frac{1}{2}t^{-\half}|Y|^2}\Big).
\end{align*}
Clearly, $\lim_{t\to\infty}c_1(t)=1$. 
The above two displays and \eqref{eq:L_infinity_bound} imply that
\begin{align}\label{eq:relative_density_upper_bound}
    \frac{\mu_s}{\rho_{at}}(x)\leq c_1(t), \quad x\in\R^d,\ s\in[0,1].
\end{align}

Use the estimate~\eqref{eq:relative_density_upper_bound} in~\eqref{eq:W_2_lower_bound_step_0} to see
\begin{align}\label{eq:W_2_lower_bound_step_1}
    \int -u (\mu*\rho_t - \nu*\rho_t)dx \leq W_2(\mu*\rho_t,\nu*\rho_t)c^\half_1(t) \Big(\int|\nabla u|^2 \rho_{at}dx\Big)^\half.
\end{align}

We now show $\int|\nabla u |^2 \rho_{at}dx$ is a good approximation of $\int|\nabla u|^2\rho_tdx$. We begin with an elementary computation, for $t>1$,
\begin{align}\label{eq:exponential_comparison}
\begin{split}
    \big|e^{-\frac{1}{2at}|x|^2}-e^{-\frac{1}{2t}|x|^2}\big|&=(1-e^{(1-a)\frac{1}{2at}|x|^2}) e^{-\frac{1}{2at}|x|^2}\\
    &\leq (a-1)\tfrac{1}{2at}|x|^2 e^{-\frac{1}{2at}|x|^2}\leq c t^{-\half}   e^{-\frac{1}{4 a t}|x|^2},
\end{split}
\end{align}
where $c > 0$ is a universal constant.
This estimate implies that 
\begin{align}\label{eq:compare_rho_at_and_rho_t}
\begin{split}
    \Big|\int|\nabla u|^2 \rho_{at}dx - a^{-\frac{d}{2}}\int|\nabla u|^2\rho_t dx\Big|\leq c t^{-\half}  \cdot 2^{d/2} \int|\nabla u|^2\rho_{2 a t}dx.
\end{split}
\end{align}

Apply H\"older's inequality to see, for $t$ large,
\begin{align}\label{eq:upp_bound_rho_delta_at}
\int |\nabla u|^2 \rho_{2 a t} d x & = \Big(\int \big(\frac{\rho_{2 a t}}{\rho_t}\big)^{4/3} \rho_t dx\Big)^{3/4} \Big(\int |\nabla u|^8 \rho_t d x\Big)^{1/4}  \nonumber\\
& = (2 a)^{- \frac{d}{2}} \Big(\int e^{\frac{2a - 1}{3} |x|^2} \fg(dx)\Big)^{3/4}  \Big(\int |\nabla u|^8 \rho_t d x\Big)^{1/4} \nonumber\\
& \leq c_d \Big(\int |\nabla u|^8 \rho_t d x\Big)^{1/4}\,.
\end{align}

Applying a change of variables, \eqref{eq:def_w} implies $\int|\nabla u|^{8}\rho_tdx = \int|\nabla w |^{8}d\fg$. Then, \eqref{eq:Dw_bounded_by_Theta_t} and Lemma \ref{Lemma:p-norm_of_Theta_t} imply that $\int|\nabla u|^{8}\rho_tdx = \mathcal{O}(t^{-4n})$. From this, \eqref{eq:compare_rho_at_and_rho_t} and \eqref{eq:upp_bound_rho_delta_at}, we obtain
\begin{align}\label{eq:difference_est_big_O}
     \int|\nabla u|^2 \rho_{at}dx = a^{-\frac{d}{2}}\int|\nabla u|^2\rho_t dx+ \mathcal{O}(t^{-n-\half}).
\end{align}

Recall \eqref{eq:formula_difference_densities}, \eqref{eq:def_w} and \eqref{eq:Lw_eq}. Changing variables and integrating by parts, we have
\begin{align*}
    \int|\nabla u|^2\rho_t dx&=\int|\nabla w|^2d\fg = \int -wLw d\fg,\\
    \int  -u (\mu*\rho_t - \nu*\rho_t)dx& = \int -u t^{-\half}\Theta_t(t^{-\half}x)\rho_t dx =\int -w\Theta_t d\fg = \int -wLw d\fg.
\end{align*}
Plug the above display and \eqref{eq:difference_est_big_O} into \eqref{eq:W_2_lower_bound_step_1} to get a lower bound
\begin{align*}
    c^{-1}_1(t)\frac{\Big(\int -wLw d\fg\Big)^2}{a^{-\frac{d}{2}}(t)\int -wLw d\fg + \mathcal{O}(t^{-n-\half})}\leq W^2_2(\mu*\rho_t,\nu*\rho_t).
\end{align*}

Since $c_1(t)$ and  $a(t)$ both converge to $1$ as $t\to \infty$ and Lemma~\ref{Lemma:asymptotics_int_wLw} implies that $\int -wLw d\fg$ is of order $t^{-n}$, we obtain Proposition~\ref{prop:p=2_lower_bound}.

\section{Lower bound for \texorpdfstring{$p=1$}{p=1}}\label{section:lower_bound}
Section~\ref{moser} establishes an upper bound on $W_p(\mu*\rho_t, \nu*\rho_t)$ valid for all $p \geq 1$.
To complete the proof of Theorem~\ref{Theorem:upper_bound}, we complement this upper bound with a lower bound on $W_1(\mu * \rho_t, \nu * \rho_t)$ of the same order.
Since $W_1 \leq W_p$ for all $p \geq 1$, this lower bound suffices to establish the desired two-sided bound on $W_p(\mu*\rho_t, \nu*\rho_t)$.

As before, it suffices to assume \eqref{eq:assumption_mean_0_EX=EY=0}.
Our technique is to employ Kantorovich-Rubinstein duality, which reads
\begin{equation*}
W_1(\mu * \rho_t, \nu * \rho_t) = \sup_{f \in \mathrm{Lip}} \int f(x) (\mu *\rho_t - \nu*\rho_t)(x) dx\,,
\end{equation*}
where the supremum is taken over all $1$-Lipschitz functions on $\R^d$.

To apply the Kantorovich--Rubinstein duality, we need to construct a suitable Lipschitz test function. For this, we need a smooth bump function $\phi:\R^d \to \R$ with the following properties:
\begin{align}\label{eq:phi_properties}
    0\leq \phi\leq 1;\quad \phi(x)=1,\ \forall x \in B_1; \quad \phi(x)=0,\ \forall x\not\in B_2;\quad |\nabla\phi|\leq 2, 
\end{align}
where $B_r$ denotes the centered Euclidean ball with radius $r>0$. Now, let us consider
\begin{align*}
    f(x)=\phi(t^{-\half}x)\Theta_t(t^{-\half}x),\quad  x \in \R^d.
\end{align*}

We first employ the following estimate, whose proof is deferred to Section~\ref{sec:lipschitz_estimate}:
\begin{align}\label{eq:f_Lipschitz}
    |\nabla f|& \leq c_{d,n}t^{-\frac{n+1}{2}}\max_{Z\in\{X,Y\}}\Big(  \E |Z|^{n+1}+t^{-\frac{n+2}{2}}\E|Z|^{2n+3}\Big) \nonumber\\
    &= c_{d,n}t^{-\frac{n+1}{2}}\max_{Z\in\{X,Y\}}\Big(  \E |Z|^{n+1}\Big)(1+o(1))
\end{align}
as $t \to \infty$.
Then, the Kantorovich--Rubinstein duality implies that
\begin{align*}
    W_1(\mu*\rho_t, \nu*\rho_t)&\geq (1-o(1))c_{\mu, \nu}t^{\frac{n+1}{2}} \int f(x)\big(\mu*\rho_t-\nu*\rho_t\big)(x)dx\\
    & = (1-o(1)) c_{\mu, \nu}t^{\frac{n}{2}} \int \phi |\Theta_t|^2d\fg\\
    & \geq (1-o(1))c_{\mu, \nu}t^{\frac{n}{2}} \int_{B_1} |\Theta_t|^2d\fg
\end{align*}
where \eqref{eq:formula_difference_densities} is used to derive the equality. To lower bound the last integral, we use \eqref{eq:Theta_t=Q+R} to see
\begin{align*}
    &\int_{B_1}|\Theta_t|^2d\fg \geq \frac{1}{2}\int_{B_1}|Q|^2d\fg - \int_{B_1}|R|^2 d\fg
\end{align*}
From this, \eqref{eq:Q_formula} and \eqref{eq:decay_rate_R}, we can derive
\begin{align*}
\begin{split}
    \int_{B_1}|\Theta_t|^2d\fg &\geq \frac{t^{-n}}{2}\int_{B_1} \bigg|\sum_{\alpha\in[n+1]}\frac{1}{\alpha!}\Big(\E X^\alpha
    - \E Y^\alpha\Big)H_\alpha(x)\bigg|^2 \fg(dx)\\
    &\quad - c_{d,n,\delta}t^{-n-1}\max_{Z\in\{X,Y\}}\Big(\E e^{\frac{\delta}{2t}|Z|^2}\Big)^2\,,
\end{split}
\end{align*}
where the fact that the Hermite polynomials form an orthogonal basis for $L^2(\R^d, \fg)$ combined with the condition $\mathbf{M}(n)$ implies that integrand on the right side is not identically zero.
We therefore obtain
\begin{equation*}
W_1(\mu*\rho_t, \nu*\rho_t) \geq (1-o(1)) c_{\mu, \nu} t^{-\frac n2}\,,
\end{equation*}
as claimed.

\section{Asymptotics for \texorpdfstring{$f$}{f}-divergences}\label{section:chi^2,KL-div}
In this section, we prove Theorems~\ref{Prop:chi^2_relative_entropy} and~\ref{Theorem:asymptotic_expansion_TV}.
In contrast to our results on the Wasserstein distances, asymptotics for $f$-divergences are significantly easier to obtain, because they are defined as explicit functions of the densities.

Our results show that, when properly rescaled, both $\chi^2(\mu*\rho_t,\nu*\rho_t)$ and $\KLD(\mu*\rho_t, \nu*\rho_t)$ possess the same limiting value, which also agrees with the limiting value of the rescaled squared $2$-Wasserstein distance.
Somewhat surprisingly, while $\chi^2$ and $\KLD$ are not symmetric in their arguments, their limiting values are symmetric.

\subsection{Exact asymptotics for \texorpdfstring{$\chi^2$}{chi-square}-divergence}

The goal is to prove \eqref{eq:exact_asymptotics_chi^2}. The $\chi^2$-divergence between $\mu*\rho_t$ and $\nu*\rho_t$ admits the following representation
\begin{align}
    \chi^2(\mu*\rho_t,\nu*\rho_t) = \int \bigg(\frac{\mu*\rho_t}{\nu*\rho_t}\bigg)^2\nu*\rho_t dx -1 = \int \frac{(\mu*\rho_t - \nu*\rho_t)^2}{\nu*\rho_t}dx.
\end{align}
To derive an upper bound, we need a lower bound for $\nu*\rho_t$. Apply Jensen's inequality to see
\begin{align*}
    \nu*\rho_t(x)=(2\pi t)^{-\frac{d}{2}}\E e^{-\frac{1}{2t}|x-Y|^2}\geq (2\pi t)^{-\frac{d}{2}} e^{-\frac{1}{2t}\E |x-Y|^2}= (2\pi t)^{-\frac{d}{2}}e^{-\frac{1}{2t}|
    x|^2 + \frac{1}{t}\langle x,\E Y\rangle - \frac{1}{2t}\E |Y|^2 }.
\end{align*}
By Cauchy--Schwarz, we have
\begin{align*}
    \langle x , \E Y\rangle \geq -\frac{1}{2t^\half}|x|^2 - \frac{t^\half}{2}|\E Y|^2.
\end{align*}
Let $a=a(t)$ be given in \eqref{eq:def_a_a(t)}. The above two displays imply that
\begin{align}\label{eq:lower_bd_nu*rho_t}
    \nu*\rho_t(x)\geq a^{-\frac{d}{2}}e^{(2t)^{-\half}|\E Y|^2 - (2t)^{-1}\E |Y|^2}\rho_{\frac{t}{a}}(x).
\end{align}
Denote the right hand side of this display by $c_{1,t}\rho_{\frac{t}{a}}(x)$, and clearly we have $\lim_{t\to\infty}c_{1,t}=1$. Hence, we obtain an upper bound:
\begin{align}\label{eq:chi^2_upper_bound}
    \chi^2(\mu*\rho_t,\nu*\rho_t)\leq c^{-1}_{1,t} \int \frac{(\mu*\rho_t - \nu*\rho_t)^2}{\rho_\frac{t}{a}}  dx.
\end{align}

For a lower bound, we use a version of \eqref{eq:mu*rho_t_upper_bound} for $\nu$ to obtain
\begin{align*}
\begin{split}
    \chi^2(\mu*\rho_t,\nu*\rho_t)\geq c^{-1}_{2,t} \int \frac{(\mu*\rho_t - \nu*\rho_t)^2}{\rho_{at}}dx.
\end{split}
\end{align*}
Here $c_{2,t}=(1+t^{-\half})\E e^{\half t^{-\half}|Y|^2}$, which converges to $1$ as $t\to\infty$. The desired result follows from this, \eqref{eq:chi^2_upper_bound} and the following lemma.

\begin{Lemma}\label{Lemma:perturbed_gaussian_asymptotics}
Suppose $\mathbf{M}(n)$ holds for some $n\in\N\cup\{0\}$. If $z=z(t)$ is a function of $t$ satisfying $|z(t)-1|\leq ct^{-\half}$ for all $t$ with some constant $c\geq0$, then it holds that
\begin{align*}
    \lim_{t\to \infty }t^{n+1}\int \frac{(\mu*\rho_t - \nu*\rho_t)^2}{\rho_{zt}}dx = \sum_{\alpha \in [n+1]}\frac{1}{\alpha!}|\E X^\alpha - \E Y^\alpha|^2.
\end{align*}
\end{Lemma}

\begin{proof}
By a similar computation in \eqref{eq:exponential_comparison}, we have
\begin{align*}
    \bigg|\frac{\rho_t(x)}{\rho_{zt}(x)}-z^{\frac{d}{2}}\bigg|\leq cz^\frac{d}{2}t^{-\half}e^{\frac{1}{4zt}|x|^2},\quad x\in\R^d.
\end{align*}
This display together with \eqref{eq:formula_difference_densities} yields, for $t$ large, 
\begin{align}\label{eq:chi^2_lower_bound_part_2}
\begin{split}
    &\bigg| \int \frac{(\mu*\rho_t - \nu*\rho_t)^2}{\rho_{zt}}dx - z^\frac{d}{2}t^{-1}\int |\Theta_t|^2d\fg\bigg|\\
    &= t^{-1}\bigg|\int|\Theta_t(t^{-\half}x)|^2\rho_t(x)\frac{\rho_t(x)}{\rho_{zt}(x)}dx - z^\frac{d}{2}\int |\Theta_t(t^{-\half}x)|^2\rho_t(x)dx\bigg| \\
    & \leq t^{-1}\int |\Theta_t(t^{-\half}x)|^2cz^\frac{d}{2}t^{-\half}e^{\frac{1}{4zt}|x|^2}\rho_t(x)dx \leq  c_dt^{-\frac{3}{2}}\int |\Theta_t(x)|^2 \rho_{\frac{2z}{2z-1}}(x)dx.
\end{split}
\end{align}
Invoke H\"older's inequality to see, for $t$ large,
\begin{align}\label{eq:chi^2_lower_bound_part_3}
\begin{split}
        \int |\Theta_t(x)|^2 \rho_{\frac{2z}{2z-1}}(x)dx &\leq c_d \bigg(\int|\Theta_t(x)|^8d\fg\bigg)^\frac{1}{4}\bigg(\int \left(\frac{\rho_{\frac{2z}{2z-1}}}{\rho_1}\right)^{\frac 4 3}d\fg\bigg)^\frac{3}{4}\\
    &\leq c_d\bigg(\int|\Theta_t|^8d\fg\bigg)^\frac{1}{4} = \mathcal{O}(t^{-n})
\end{split}
\end{align}
where the last identity follows from Lemma \ref{Lemma:p-norm_of_Theta_t}.

To compute $\int|\Theta_t|^2d\fg$, let us recall \eqref{eq:Theta_t=Q+R} and \eqref{eq:QR_orthogonal}. These imply
\begin{align*}
    \int|\Theta_t|^2d\fg = \int|Q|^2d\fg + \int|R|^2d\fg.
\end{align*}
From this decomposition, \eqref{eq:int_Q^2} and \eqref{eq:decay_rate_R}, we obtain
\begin{align}\label{eq:exact_asymptotics_Theta_t_square}
    \lim_{t\to\infty} t^n\int|\Theta_t|^2d\fg = \sum_{\alpha \in [n+1]}\frac{1}{\alpha!}\Big|\E X^\alpha - \E Y^\alpha\Big|^2.
\end{align}

The proof is complete by combining \eqref{eq:chi^2_lower_bound_part_2}, \eqref{eq:chi^2_lower_bound_part_3} and \eqref{eq:exact_asymptotics_Theta_t_square}.

\end{proof}

\subsection{Exact asymptotics for relative entropy} 
In this section, we prove \eqref{eq:exact_asymptotics_KL-div}. For simplicity, let us write $f=\mu*\rho_t$ and $g=\nu*\rho_t$. 
Using the Taylor expansion
\begin{align*}
\log(x+1)= x + \int_0^1\frac{(s-1)x^2}{(1+sx)^2}ds, \quad\quad x>-1,
\end{align*}
we have
\begin{align*}
\log\Big(\frac{f-g}{g}+1\Big)= \frac{f-g}{g}+\int_0^1\frac{(s-1)(f-g)^2}{(sf+(1-s)g)^2}ds.
\end{align*}
Therefore, we obtain
\begin{align}\label{eq:D_KL_decomposition}
\KLD(f\|g)=\chi^2(f,g)+\int_0^1(s-1) \int_{\R^d}\frac{(f-g)^2}{(sf+(1-s)g)^2}fdxds.
\end{align}

Since we have already established the asymptotic behavior of $\chi^2$, we focus on the second term. Due to \eqref{eq:mu*rho_t_upper_bound} and \eqref{eq:lower_bd_nu*rho_t} for both $\mu$ and $\nu$, there are $c_{1,t}$ and $c_{2,t}$, both converging to $1$ as $t\to \infty$, such that
\begin{align*}
    c_{1,t}\rho_{\frac{t}{a}}(x) \leq f(x),\ g(x)\leq c_{2,t}\rho_{at}(x),\quad x\in\R^d. 
\end{align*}
with $a=a(t)$ given in \eqref{eq:def_a_a(t)}. Therefore, we have
\begin{align*}
    \frac{c_{1,t}}{c_{2,t}^2}\int \frac{(f-g)^2}{\rho_{at}^2}\rho_{\frac{t}{a}}dx\leq \int\frac{(f-g)^2}{(sf+(1-s)g)^2}fdx\leq \frac{c_{2,t}}{c_{1,t}^2}\int \frac{(f-g)^2}{\rho_{\frac{t}{a}}^2}\rho_{at}dx.
\end{align*}
After computing $\rho_{\frac{t}{a}}/\rho^2_{at}$ and $\rho_{at}/\rho^2_{\frac{t}{a}}$, one can see the above becomes
\begin{align*}
     \frac{c_{1,t}}{c_{2,t}^2}(2a^2-a^4)^\frac{d}{2}\int \frac{(f-g)^2}{\rho_{\frac{a}{2-a^2}t}}dx\leq \int\frac{(f-g)^2}{(sf+(1-s)g)^2}fdx\leq \frac{c_{2,t}}{c_{1,t}^2}\Big(\frac{2a^2-1}{a^4}\Big)^\frac{d}{2}\int \frac{(f-g)^2}{\rho_{\frac{a}{2a^2-1} t}}dx.
\end{align*}
Note that both the upper bound and the lower bound are independent of $s$. These along with Lemma \ref{Lemma:perturbed_gaussian_asymptotics} imply
\begin{align*}
    \lim_{t\to \infty}t^{n+1}\int_0^1(s-1) \int_{\R^d}\frac{(f-g)^2}{(sf+(1-s)g)^2}fdxds = -\frac{1}{2} \sum_{\alpha \in [n+1]}\frac{1}{\alpha!}\Big|\E X^\alpha - \E Y^\alpha\Big|^2.
\end{align*}
This together with \eqref{eq:D_KL_decomposition} and \eqref{eq:exact_asymptotics_chi^2} finishes the proof.
\subsection{Exact asymptotics for total variation distance}
Recall \eqref{eq:def_eta} and we have
\begin{align*}
    &d_{\mathrm{TV}}(\mu*\rho_t,\nu*\rho_t)=\frac{1}{2}\int\big|\mu*\rho_t-\nu*\rho_t\big|dx\\ 
    &= \frac{1}{2}\int|t^{-\half}\Theta_t(t^{-\half}x)|\rho_t(x)dx = \frac{1}{2}t^{-\half}\int |\Theta_t|d\fg.
\end{align*}
Due to \eqref{eq:def_eta} and \eqref{eq:eta_expansion}, for fixed $n$, we write $\Theta_t = Q +R$ with
\begin{align*}
    Q(x)&=t^\half \sum_{j=0}^{n+1}\Big(\E a_{j}(x,t^{-\half}X)-\E a_{j}(x,t^{-\half}Y)\Big)\\
    & = \sum_{j=1}^{n+1}t^{-\frac{j-1}{2}}\sum_{\alpha\in[j]}\frac{1}{\alpha!}\Big(\E X^\alpha - \E Y^\alpha\Big)H_\alpha(x),
\end{align*}
and $R$ as in \eqref{eq:R_formula}. The triangle inequality implies that
\begin{align*}
    \bigg|\int|\Theta_t|d\fg - \int|Q|d\fg\bigg| \leq \int |R|d\fg.
\end{align*}
Applying Lemma~\ref{lemma:r_n+1_integral_est} with $m=n+1$ and $p=1$, we obtain the desired result.

\section{Auxiliary results}

\subsection{Proof of Lemma~\ref{lemma:r_n+1_integral_est}}\label{proof_of_integral_est}
\begin{proof}

Formula \eqref{eq:r_n+1_formula}  implies that
\begin{align*}
\begin{split}
    &\int |t^\half\E r_{m+1}(x, t^{-\frac{1}{2}}Z)|^p\fg(dx)\\
    &= \frac{m+1}{t^\frac{mp}{2}}\int\bigg|\int_0^1(1-s)^m\sum_{\alpha\in[m+1]}\frac{1}{\alpha!}\E Z^\alpha  H_\alpha (x-s t^{-\half}Z)\eta(x,st^{-\half}Z)ds\bigg|^p\fg(dx)\\
    &\leq \frac{c_{m,p}}{t^\frac{mp}{2}}\sum_{\alpha\in[m+1]}\int\Big|\int_0^1\E Z^\alpha  H_\alpha (x-s t^{-\half}Z)\eta(x,st^{-\half}Z)ds\Big|^p\fg(dx)\\
    &\leq \frac{c_{m,p}}{t^\frac{mp}{2}}\sum_{\alpha\in[m+1]}\Bigg(\int_0^1\E \bigg(\int\Big| Z^\alpha  H_\alpha (x-s t^{-\half}Z)\eta(x,st^{-\half}Z)\Big|^p\fg(dx)\bigg)^\frac{1}{p} ds\Bigg)^p,
\end{split}
\end{align*}
where in the last inequality we used Minkowski's integral inequality.

Then, we estimate the integral with respect to $\fg$. Recall $\eta$ in \eqref{eq:def_eta}. Due to $\alpha\in[m+1]$, we have $|Z^\alpha |\leq |Z|^{m+1}$. Since $H_\alpha $ is a polynomial of degree $m+1$ as evident in \eqref{eq:def_H_alpha}, one can see that $|H_\alpha (x)|\leq c_{d,m}(1+|x|^{m+1})$. Therefore, we obtain
\begin{align*}
    &\int\Big| Z^\alpha  H_\alpha (x-s t^{-\half}Z)\eta(x,st^{-\half}Z)\Big|^p\fg(dx)\\
    &\leq c_{d,m,p} |Z|^{p(m+1)}\int \Big(1+|x-st^{-\half}Z|^{p(m+1)}\Big)e^{\langle x, pst^{-\half}Z\rangle-\frac{p}{2}|st^{-\half}Z|^2}e^{-\half|x|^2}dx\\
    &= c_{d,m,p} |Z|^{p(m+1)} \bigg(\int \Big(1+|x-st^{-\half}Z|^{p(m+1)}\Big)e^{-\half|x-pst^{-\half}Z|^2}dx\bigg) e^{\frac{p^2-p}{2}|st^{-\half}Z|^2}\\
    &= c_{d,m,p} |Z|^{p(m+1)}\Big(1+|st^{-\half}(p-1)Z|^{p(m+1)}\Big)e^{\frac{p^2-p}{2}|st^{-\half}Z|^2}.
\end{align*}
The above two displays yield, for $\delta>1$,
\begin{align*}
\begin{split}
    &\int |t^\half\E r_{m+1}(x, t^{-\frac{1}{2}}Z)|^p\fg(dx)\\
    &\leq c_{d,m,p}t^{-\frac{mp}{2}}\bigg(  \E\Big(|Z|^{m+1}+t^{-\frac{m+1}{2}}|(p-1)Z|^{2m+2}\Big)e^{\frac{p-1}{2t}|Z|^2}\bigg)^p\\
    &\leq c_{d,m,p,\delta,\beta}t^{-\frac{mp}{2}}\Big( \E e^{\frac{\delta(p-1)}{2t}|Z|^2} \Big)^p,\quad t>\tfrac{\delta(p-1)}{2\beta}.
\end{split}
\end{align*}
For $p=1$, it can be checked that the above is valid as long as $\E|Z|^{m+1}<\infty$.

\end{proof}

\subsection{Proof of \eqref{eq:QR_orthogonal}}\label{QR_proof}
\begin{proof}

Due to \eqref{eq:Q_formula}, it suffices to show 
\begin{align}\label{eq:orthogonality_condition}
\int H_\alpha (\Theta_t - Q) d\fg=0, \quad \text{for all } \alpha \in [n+1].
\end{align}
Using \eqref{eq:def_eta} and changing variables, we have, for $\alpha \in [n+1]$,
\begin{align*}
\int H_\alpha \Theta_t d\fg & = \frac{t^\half}{(2\pi )^\frac{d}{2}}\int \E H_\alpha(x)\Big(e^{-\half|x-t^{-\half}X|^2}-e^{-\half|x-t^{-\half}Y|^2}\Big)dx\\
& = \frac{t^\half}{(2\pi )^\frac{d}{2}} \int \E \big(H_\alpha (x+t^{-\half}X)-H_\alpha (x+t^{-\half}Y)\big) e^{-\half|x|^2}dx.
\end{align*}
We claim that
\begin{align}\label{eq:polynomail_identity}
\E \big(H_\alpha (x+t^{-\half}X)-H_\alpha (x+t^{-\half}Y)\big)= t^{-\half (n+1)} (\E X^\alpha - \E Y^\alpha).
\end{align}
This immediately gives
\begin{align*}
    \int H_\alpha \Theta_t d\fg = t^{-\frac{n}{2}}(\E X^\alpha - \E Y^\alpha).
\end{align*}
On the other hand, by Lemma \ref{lemma:hermite_poly_orthognoality}, we have 
\begin{equation*}
    \int H_\alpha Q d\fg = t^{-\frac{n}{2}}(\E X^\alpha - \E Y^\alpha).
\end{equation*}
The above two displays give us \eqref{eq:orthogonality_condition}.

To show \eqref{eq:polynomail_identity}, we introduce the following notation. For $\beta \in \N^d$, we write $\beta \leq \alpha$ if $\beta_i\leq \alpha_i$ for all $i=1,2,\dots,d$. If $\beta \leq \alpha$ and $\beta \neq \alpha$, we write $\beta<\alpha$. Lastly, let $|\beta|= \sum_{i}\beta_i$. 

By \eqref{eq:def_hermite_poly} and \eqref{eq:def_H_alpha}, we know that $H_\alpha(x)$ is a polynomial of degree $|\alpha|=n+1$ and the leading order term is $x^\alpha$. Hence, there are coefficients $c_\beta$ for $\beta\leq \alpha$ such that the left hand side of \eqref{eq:polynomail_identity} admits the following expansion
\begin{align*}
    \E \big(H_\alpha (x+t^{-\half}X)-H_\alpha (x+t^{-\half}Y)\big) = \sum_{\beta\leq \alpha}c_\beta x^{\alpha-\beta}t^{-\half |\beta|}(\E X^\beta - \E Y^\beta).
\end{align*}
If $\beta<\alpha$, then $|\beta|\leq n$. Hence, by the assumption $\mathbf{M}(n)$, we must have $\E X^\beta = \E Y^\beta$ for all $\beta <\alpha$. Therefore, the only term that does not vanish on the right of the above display is $c_\alpha t^{-\half(n+1)}(\E X^\alpha - \E Y^\alpha)$ and $c_\alpha =1$ as evident from \eqref{eq:def_hermite_poly}. This verifies \eqref{eq:polynomail_identity} and completes the proof.

\end{proof}

\subsection{Proof of \eqref{eq:f_Lipschitz}}\label{sec:lipschitz_estimate}
To show $f$ is uniformly Lipschitz and figure out its Lipschitz coefficient, we start to estimate, using \eqref{eq:phi_properties},
\begin{align*}
    |\nabla f(x)|\leq t^{-\half}\Big(2\big|\Theta_t(t^{-\half}x)\big|+ \big|\nabla\Theta_t(t^{-\half}x)\big|\Big)\Ind{|t^{-\half}x|\leq 2}.
\end{align*}
To bound $\Theta_t(x)$ for $|x|\leq 2$, due to Lemma \ref{Lemma:Theta_formula_for_matching_moments}, we only need to estimate using \eqref{eq:r_n+1_formula}, for $|x|\leq 2$ and $Z\in\{X,Y\}$,
\begin{align*}
    t^\half\E|r_{n+1}(x,t^{-\half}Z)|&\leq c_{d,n} t^{-\frac{n}{2}}\int_0^1\E|Z|^{n+1}\bigg|\sum_{\alpha\in[n+1]}H_\alpha(x-st^{-\half}Z)\eta(x, st^{-\half}Z)\bigg|ds\\
    &\leq c_{d,n}t^{-\frac{n}{2}}\E|Z|^{n+1}\bigg(1+|x|^{n+1}+|t^{-\half}Z|^{n+1}\bigg)e^{\frac{1}{2}|x|^2}\\
    &\leq c_{d,n}t^{-\frac{n}{2}}\Big(  \E |Z|^{n+1}+t^{-\frac{n+1}{2}}\E|Z|^{2n+2}\Big).
\end{align*}
Here, to derive the second inequality, we used the fact that $H_\alpha$ is a polynomial of degree $n+1$, and also the formula of $\eta$ in \eqref{eq:def_eta}.

Again by Lemma \ref{Lemma:Theta_formula_for_matching_moments}, to bound $|\nabla\Theta_t(t^{-\half}x)|$, we only need to show, for $|x|\leq 2$ and $Z\in\{X,Y\}$,
\begin{align*}
    t^\half\E\big|\nabla r_{n+1}(x,t^{-\half}Z)\big|&\leq c_{d,n} t^{-\frac{n}{2}}\int_0^1\E|Z|^{n+1}\bigg|\sum_{\alpha\in[n+1]}\nabla\Big(H_\alpha(x-st^{-\half}Z)\eta(x, st^{-\half}Z)\Big)\bigg|ds\\
    &\leq c_{d,n}t^{-\frac{n}{2}}\E|Z|^{n+1}\bigg(\big(1+|x|^{n}+|t^{-\half}Z|^{n}\big)e^{\frac{1}{2}|x|^2}\\
    &\quad\quad\quad\quad + \big(1+|x|^{n+1}+|t^{-\half}Z|^{n+1}\big)|t^{-\half}Z|e^{\half|x|^2}\bigg)\\
    &\leq c_{d,n}t^{-\frac{n}{2}}\Big(  \E |Z|^{n+1}+t^{-\frac{n+2}{2}}\E|Z|^{2n+3}\Big).
\end{align*}
In the second inequality, we applied the product rule of differentiation, and again used the fact of $H_\alpha$ being a polynomial of degree $n+1$ and the definition of $\eta$ in \eqref{eq:def_eta}.

From all these estimates, we derive that
\begin{align*}
    |\nabla f|\leq c_{d,n}t^{-\frac{n+1}{2}}\max_{Z\in\{X,Y\}}\Big(  \E |Z|^{n+1}+t^{-\frac{n+2}{2}}\E|Z|^{2n+3}\Big)\,,
\end{align*}
as claimed.

\bibliographystyle{abbrvnat}

\end{document}